\documentclass{article}
\usepackage[utf8]{inputenc}
\usepackage{johd}
\usepackage[all]{xy}
\usepackage{amsmath,amsthm,amsfonts}

\usepackage{graphicx}
\usepackage{enumerate} 
\newcommand{\Cent}{\mathrm{Cent}}
\theoremstyle{definition}
\newtheorem{theorem}{Theorem}[section]

\newtheorem{lemma}[theorem]{Lemma}
\newtheorem{corollary}[theorem]{Corollary}

\newtheorem{ex}[theorem]{Example}
\newtheorem{remark}[theorem]{Remark}
\newtheorem{definition}[theorem]{Definition}
\newtheorem*{ex*}{Example}
\newtheorem*{cor*}{Corollary}
\newcommand{\Z}{\mathbb{Z}}

\title{Homomorphism Conjugacy versus Centralizer Actions in the Symmetric Group}

\author{Suleika Norrbom}
\date{}

\begin{document}
\maketitle
\bibliographystyle{plainnat}
\begin{abstract} 
\noindent We explore when generator-conjugate homomorphisms are conjugate and when element-conjugate homomorphisms are conjugate from abelian or dihedral groups to the symmetric group. We completely determine when such homomorphisms are conjugate in the case where the source group has two generators by studying centralizer actions in the target group. \end{abstract}
\section{Introduction}
Element-conjugate linear representations of finite degree are conjugate if the underlying field is algebraically closed and of characteristic zero. We are motivated by exploring the relationship between element-conjugacy and homomorphism conjugacy for other target groups. M. Larsen initiated this discussion of the relationship between element-conjugate homomorphisms and conjugate homomorphisms for Lie groups in his papers from 1994 \cite{larsen94} and 1996 \cite{larsen96}. In the papers, a group $G$ is acceptable if a homomorphism from a group $\Gamma$ to $G$ is determined up to conjugacy by the conjugacy classes of the elements $\phi(\gamma)$. Larsen made progress toward the classification of acceptable subgroups of the general linear group where  all algebraic groups are assumed to be connected even if their associated Lie groups are not. Weidner later extended this  work by using pseudocharacters to answer more questions about acceptable algebraic groups \cite{weidner20}.

We focus on a discrete variation of this problem by determining under which conditions some generator-conjugate or element-conjugate homomorphisms to the symmetric group are conjugate. 

\subsection{Notation and Important Definitions}

\subsubsection{Notation}
Let $G$ be a group and $x$ be an element in $G$, $\Cent_G(x)$ denotes the centralizer of $x$ in $G$. We have that $S_n$ is the symmetric group of the set $\{1,2,...,n\}$ while $S_X$ is the symmetric group of the set $X$. Let $\sigma \in S_n$, then $\mathbf{Fix}(\sigma)$ is the set of fix points of $\sigma$ in $\{1,2,..,n\}$. We set $X_{\sigma} = \{1,2,...,n\} \setminus \mathbf{Fix}(\sigma)$. Let $\sigma_1,\sigma_2 \in S_n$. We set $X_{\sigma_1,\sigma_2}=\{1,2,...,n\}\setminus \mathbf{Fix}(\sigma_1) \cup \mathbf{Fix}(\sigma_2).$
\subsubsection{Definitions} 
Let $G$ and $H$ be groups and let $\varphi,\psi$ be homomorphisms from $G$ to $H$. 
\begin{definition}[\textbf{Generator-conjugate}]
Assume $G$ is generated by the elements $g_1,g_2,...,g_n \in G$ and that for each $g_i$, there exists some $h \in H$ so that 
\begin{equation}
h\varphi(g_i)h^{-1} = \psi(g_i),
\end{equation}
then $\varphi, \psi$ are generator-conjugate. 
\end{definition}

\begin{definition}[\textbf{Element-conjugate}] If for every $g \in G$, there exists some $h \in H$ so that 
\begin{equation}
h\varphi(g)h^{-1} = \psi(g), 
\end{equation}
we say $\varphi, \psi$ are element-conjugate. 
\end{definition}

\begin{definition}
If there exists some $h \in H$ so that for all $g \in G$: 
\begin{equation}
h\varphi(g)h^{-1} = \psi(g), 
\end{equation}
then $\varphi, \psi$ are \textbf{conjugate}.
\end{definition}

\subsection{Summary}

We begin by studying actions induced by centralizers in the symmetric group. Assume $\sigma \in S_n$. The first action is due to the invariance of cycle type under conjugation in the symmetric group and gives a direct product decomposition of the centralizer $\Cent_{S_n}(\sigma)$(Lemma \ref{directproduct}). Let $\sigma^\prime \in S_n$ be the product of $k$ disjoint cycles of length $d$,$\; \tau_1\tau_2\ldots\tau_k$. The centralizer of $\sigma^\prime$ acts transitively by conjugation on the set $\{\tau_1,\tau_2,\ldots,\tau_k\}$ (Lemma \ref{actioncycles}). 

These two actions to a large extent determine whether two elements $\pi,\pi^\prime\in S_n$ are conjugate by some element $\rho$ in the centralizer of $\sigma$. 

Let $\varphi,\psi$ be generator-conjugate homomorphisms from a group $G$ to the symmetric group. We assume that both homomorphisms send some generator $g$ to the same point $\sigma$ in $S_n$. If $\varphi, \psi$ are conjugate as homomorphisms, there must exist some $\rho \in \Cent_{S_n}(\sigma)$ so that $\rho\varphi\rho^{-1}=\psi$. We are concerned with determining when such an element $\rho$ exists. Exploring centralizer actions in relation to homomorphism conjugacy in this way, we obtain our main results. 
\subsubsection{Main results of the paper}
Assume $\varphi,\psi$ are homomorphisms from an abelian group with two generators or a dihedral group, to the symmetric group $S_n$. In these cases we: 
\begin{itemize}
    \item Describe under which conditions generator-conjugate homomorphisms are conjugate (Theorem \ref{abeliantheorem} and Theorem \ref{dihedraltheorem}).
    \item Detail the relationship between element-conjugate homomorphisms and conjugate homomorphisms (Corollary \ref{localglobalabelian} and Corollary \ref{localglobaldihedral}). 
\end{itemize}

\subsection{Possible Future Directions}
The dihedral case describes the relationship between generator-conjugacy, element-conjugacy, and
conjugacy for homomorphisms from a finite reflection group with two generators to the symmetric
group. A natural next step could be attempting to generalize some of these results to other finite
reflection groups as source groups. In Lemma \ref{actioncycles}, we prove that the centralizer of an element $\sigma=\tau_1\tau_2..\tau_k$, the product of $k$-cycles of length $d$, acts transitively on the set $\{\tau_1,\tau_2,..,\tau_k \}$. However, we later note that there is also a larger subgroup $H_{\sigma}$ that acts transitively on the set  $\{ \{\tau_1\},...\{\tau_k \}\}$, where $\{\tau_i\}$ is the set of points in the cycle $\tau_i$ (Lemma \ref{actiondihedral}): 
\begin{equation}
H_{\sigma} = \{ x \in S_n | x\sigma x^{-1} = \sigma^k, k \in \mathbb{Z} \}
\end{equation}
Note that any finite reflection group $S$ can be represented in the form: 
\begin{equation}
S = \langle s_1,s_2,...,s_n \;| \; s_i^2=1, (s_is_j)^{m_{i,j}}=1 , \; m_{i,j} \in \mathbb{Z} \rangle,
\end{equation}
and that $s_i,s_j \in H_{s_is_j}$:
\begin{equation}
s_i(s_is_j)s_i = s_j(s_is_j)s_j = s_js_i = (s_is_j)^{-1}.
\end{equation}
A way to work out the relations between element-conjugacy and conjugacy for homomorphisms to the symmetric group from an arbitrary finite reflection group might be through this connection to actions, in a similar way to the abelian or dihedral case we have already covered.

Further, we could also let the target group itself be any  finite reflection group and try to further generalize the results to this setting. 

Another potential route is by first noting that $S_n$ is the Weyl group of $\mathrm{GL}(n,\mathbb{C})$. It could then be interesting to explore the connection between element-conjugacy versus conjugacy for homomorphisms to each group. This could be expanded into attempting to explain this relation for other connected, reductive groups $G$ and their corresponding Weyl groups. 

\section{Centralizer Actions in the Symmetric Group}

Let $\sigma$ be an arbitrary permutation in $S_n$. We begin by defining the notion of a non-trivial centralizer. 

\begin{definition} [Non-trivial Centralizer] We call the centralizer of $\sigma$ in the symmetric group $S_{X_{\sigma}}$ where $X_{\sigma}=\{1,2,...,n\}\setminus \mathbf{Fix}(\sigma)$ the non-trivial centralizer of $\sigma$ and denote it $\mathrm{Cent}_0(\sigma)$.
\end{definition}

We may rewrite $\sigma$ as a product of disjoint cycles and group the cycles of the same length together such that 
\begin{equation}
\sigma = \sigma_1\sigma_2...\sigma_l,
\end{equation}
where each $\sigma_i$ is the product of $k_i$ disjoint cycles of length $d_i$ and all the integers $d_1,...,d_l$ are distinct. The first important centralizer action we observe is induced by the invariance of cycle type under conjugation in $S_n$. 

\begin{lemma}[Direct Product Decomposition of Centralizer]
\label{directproduct}
The centralizer of $\sigma=\sigma_1\sigma_2...\sigma_l$ has a direct product decomposition in $S_n$
\begin{equation}
\mathrm{Cent}_{S_n}(\sigma)=S_{\mathbf{Fix}(\sigma)}\times\mathrm{Cent}_0(\sigma_1)\times\mathrm{Cent}_0(\sigma_2)\times...\times\mathrm{Cent}_0(\sigma_l)
\end{equation}
\end{lemma}

\begin{proof} Assume $\pi \in \mathrm{Cent}_{S_n}$. Then
\begin{equation}
\pi\sigma_1\pi^{-1}\pi\sigma_2\pi^{-1}...\pi\sigma_l\pi^{-1}=\sigma_1\sigma_2...\sigma_l.
\end{equation}
By invariance of cycle type under conjugation, it must hold that
\begin{equation}
\pi\sigma_i\pi^{-1}=\sigma_i, \; \; 1\leq i \leq l.
\end{equation}
Every point in a cycle of length $d_i$ in $\sigma_i$ must be sent to some cycle of length $d_i$ which also lies in $\sigma_i$ by construction.
\end{proof}

This direct product decomposition of the centralizer motivates restricting our further study of centralizer actions in $S_n$ to the case where $\sigma$ is the product of $k$ disjoint cycles of length $d$. We assume $\sigma$ satisfies this condition for the remainder of this entire section.  

The second important action for our results is presented in the following lemma. 

\begin{lemma} 
\label{actioncycles}
$\mathrm{Cent}_0(\sigma)$ acts transitively by conjugation on the set $\{\tau_1,\tau_2,...,\tau_k\}$.
\end{lemma}
\begin{proof}
The identity element acts trivially on all elements in the set. We show that the orbit of $\tau_1$ is the entire set. We can set $\tau_1=(12\cdots d)$ and $\tau_i=(i_1i_2...i_d)$ and define a permutation:
\begin{equation}
\pi(j)=i, \pi(i_j) =j; 1 \leq j \leq d, \; \mathbf{Fix}(\pi) = X_{\sigma} \setminus \{1,2,...,n\}\cup\{i_1,i_2,...,i_d\}
\end{equation}
Clearly, $\pi \in \mathrm{Cent}_0(\sigma)$ and $\pi\tau_1\pi^{-1}= \tau_i$. Next, assume there exists some element $\rho \in \mathrm{Cent}_0(\sigma)$ so that $\rho\tau_1\rho^{-1} \notin \{\tau_1,\tau_2,...,\tau_k\}$. Since we assume that the cycles $\tau_1,\tau_2,..,\tau_k$ are disjoint, this implies that $\rho\tau_1...\tau_k\rho^{-1} \ne \tau_1...\tau_k$ which is a contradiction. 
\end{proof}

\begin{corollary} 
\label{kisnormal}
The subgroup $K=\langle \tau_1, \tau_2, \tau_3,...,\tau_k \rangle$ is normal in $\mathrm{Cent}_0(\sigma)$.
\end{corollary}

Note that the associated permutation representation of the action of $\mathrm{Cent}_0(\sigma)$ on $\{ \tau_1,...,\tau_k \}$ is precisely given by the homomorphism whose kernel is $K$:
\begin{equation}
\mathrm{Cent}_0(\sigma) \longrightarrow \mathrm{Cent}_0(\sigma)/K \cong S_k
\end{equation}
Each element $\pi \in \mathrm{Cent}_0(\sigma)$ may therefor by its image under $\mathrm{Cent}_0(\sigma)$ be associated with some element in $S_k$, describing its action on the cycles $\tau_1,...,\tau_k$.

\begin{remark}
Let $\pi \in \mathrm{Cent_0}(\sigma)$. The cycle type of $\pi$ under the homomorphism $\mathrm{Cent}_0(\sigma)/K$ is invariant under conjugation in $\mathrm{Cent}_0(\sigma)$. 
\end{remark} 

This is just the result that cycle type is invariant under conjugation in the symmetric group. However, this remark is significant for our discussion about homomorphism conjugacy since two elements $\pi,\pi' \in \mathrm{Cent}_0(\sigma)$ may be conjugate in $S_n$ but have distinct cycle types under $\mathrm{Cent}_0(\sigma)/K$. 

\begin{ex}
Assume $\sigma = (12)(34)(56)(78)$ and set $\pi=(1324)(5768)$ and $\pi'=(1357)(2468)$. The image of $\pi$ under
 $\mathrm{Cent}_0(\sigma)/K$ is $(12)(34)$ while the image of $\pi'$ under $\mathrm{Cent}_0(\sigma)$ is $(1234)$
\end{ex}

We now split our focus on centralizer actions in the symmetric group in two separate directions, based on the two source groups (abelian or dihedral) we are interested in exploring homomorphism conjugacy for.

In the abelian case, we would like to determine when for two elements $\pi,\pi' \in \mathrm{Cent}_0(\sigma)$, there exists some $\rho \in \mathrm{Cent}_0(\sigma)$ so that 
\begin{equation}
\rho\pi'\rho^{-1}=\pi.
\end{equation}

In the dihedral case, we are interested in elements $\pi,\pi'$, conjugate in $S_n$, where
\begin{equation}
\pi^2=(\pi')^2=1 \; \; \mathrm{and} \;\; \pi\sigma\pi^{-1}=(\pi')\sigma(\pi')^{-1}=\sigma^{-1}.
\end{equation}
In particular, we wish to determine when there exists some $\rho \in \mathrm{Cent_0}(\sigma)$ such that 
\begin{equation}
\rho\pi'\rho^{-1} = \pi.
\end{equation}

\subsection{Centralizer Actions for Abelian Source}

Assume as before that $\sigma \in S_n$ is the product of $k$ disjoint cycles of length $d$:
\begin{equation}
\tau_1\tau_2...\tau_k.
\end{equation}
    From now on, we also assume that $\mathbf{Fix}(\sigma)= \emptyset$. The primary goal of this subsection is to completely determine under which conditions two elements $\pi,\pi' \in \mathrm{Cent}_0(\sigma)$ are conjugate in $\mathrm{Cent}_0(\sigma)$. It is sufficient to explore this problem for the case where the image of $\pi$ is a single cycle of length $m$ under $\mathrm{Cent}_0(\sigma)/K$: 
\begin{equation}
\mathrm{Cent}_0(\sigma)/K: \pi \mapsto (12\cdots m).
\end{equation}

\begin{remark}
\label{frstremark}
If $\pi,\pi'$ are conjugate in $\mathrm{Cent}_0(\sigma)$, they have the same cycle type under $\mathrm{Cent}_0(\sigma)/K$.
\end{remark}
We have that this follows immediately from the invariance of cycle type under conjugation in the symmetric group. This in conjunction with that $K = \langle \tau_1,\tau_2,..,\tau_k \rangle$ is normal in $\mathrm{Cent}_0(\sigma)$ by Corollary \ref{kisnormal} means we may write any element in the non-trivial centralizer as a product of some element in $K$ and some element not in $K$, whose preimage of the identity under $\Cent_0(\sigma)/K$ is trivial. 

\begin{remark} 
\label{decomp}
We can rewrite the elements $\pi,\pi'$ such that 
\begin{equation}
\pi = \mu\pi_0, \pi=\mu'\pi_0', 
\end{equation}
where $\mu, \mu' \in K$ and $\pi_0,\pi_0'$ permutes no cycles trivially, i.e either all points in a cycle $\tau_i$ are fixed or sent to the points in some distinct cycle $\tau_j$. 
\end{remark}

This in conjunction with the action on the set $\{\tau_1,...,\tau_k\}$ (Lemma \ref{actioncycles}) results in an additional condition for conjugacy in $\mathrm{Cent}_0(\sigma)$: 

\begin{remark}
\label{sndremark}
There are integers $z_1,z_2,...,z_k$ and $n_1,n_2,...,n_k$ such that 
\begin{equation}
\mu = \tau_1^{z_1}\tau_2^{z_2}...\tau_k^{z_k} \; \; \mathrm{and} \; \; \mu' = \tau_1^{n_1}\tau_2^{n_2}...\tau_k^{n_k}.
\end{equation}
If $\pi,\pi'$ are conjugate in $\mathrm{Cent}_0(\sigma)$, then the lists of integers $z_1,z_2,...,z_k$ and $n_1,n_2,...,n_k$ are identical up to rearrangement. 
\end{remark}

We now make an important observation about $\pi_0$ in the decomposition $\pi = \mu\pi_0$ as in Remark \ref{decomp}. 

\begin{lemma} We have that
\begin{equation}
\pi_0^m = \tau_1^z\tau_2^{z}...\tau_m^{z},
\end{equation}
for some integer $z$.
\end{lemma}
\begin{proof}
It is clear that $\pi_0^m \in \ker(\mathrm{Cent}_0(\sigma)/K)$. Then $\pi_0^m = \tau_1^{z_1}\tau_2^{z_2}...\tau_m^{z_m}$ for some integers $z_1,z_2,..,z_m$. Assume there exists a pair of integers $i<j$ such that $z_i \ne z_j$. This implies that $\pi_0^{j-i}$ does not commute with $\pi_0^m$ which is a contradiction. 
\end{proof}

This lemma makes it possible to define the last necessary condition for conjugacy in $\mathrm{Cent}_0(\sigma)$. 

\begin{remark}
\label{trdremark}
We rewrite the elements $\pi,\pi'$ such that 
\begin{equation}
\pi = \mu\pi_0, \pi=\mu'\pi_0', 
\end{equation}
as in Remark \ref{decomp}. If $\pi,\pi'$ are conjugate in $\mathrm{Cent}_0(\sigma)$, then $\pi_0^m=\tau_1^{z}\tau_2^{z}...\tau_m^{z}$ and $(\pi_0')^m=\tau_{i_1}^z\tau_{i_2}^z...\tau_{i_m}^z$ for an integer $z$.
\end{remark}
\begin{proof}
By that $K$ is normal in $\mathrm{Cent}_0(\sigma)$, we must have that $\pi_0,\pi_0'$ are conjugate in $\mathrm{Cent}_0(\sigma)$. Assume that $\pi_0^m=\tau_1^z...$ and $(\pi_0')^m=\tau_{i_1}^{z'}...$ for some distinct integers $z, z'$. Then $\tau_1^z, \tau_{i_1}^{z'}$ are conjugate in $\mathrm{Cent}_0(\sigma)$. This implies that the orbit of $\tau_1$ under the action of $\mathrm{Cent}_0(\sigma)$ is not the set $\{\tau_1,\tau_2,..,\tau_k\}$ which is a contradiction by Lemma \ref{actioncycles}. 
\end{proof}

We are ready to define precisely when $\pi, \pi'$ are conjugate in $\mathrm{Cent}_0(\sigma)$. 

\begin{lemma} 
\label{importantabelianlemma}
Let $\pi,\pi'$ be two elements in the non-trivial centralizer of $\sigma=\tau_1...\tau_k$ where 
\begin{equation}
\mathrm{Cent}_0(\sigma)/K : \pi \mapsto (12\cdots m),
\end{equation}
recalling that $K= \langle \tau_1,\tau_2,...,\tau_k \rangle$. 
We may rewrite $\pi, \pi'$ by Remark \ref{decomp} so that $\pi=\mu\pi_0$ and $\pi'=\mu'\pi_0'$ where $\mu, \mu' \in K$ and $\pi_0,\pi_0' \notin K$. Then $\mu_0=\tau_1^{z_1}...\tau_k^{z_k}$ for some integers $z_1,z_2,..,z_k$ and $\mu_0=\tau_1^{n_1}...\tau_k^{n_k}$ for some integers $n_1,n_2,..,n_k$. 
The elements $\pi, \pi'$ are conjugate in $\mathrm{Cent}_0(\sigma)$ if and only if:
\begin{enumerate}[(i)]
\item The image of $\pi'$ under $\mathrm{Cent}_0(\sigma)/K$ is also an $m$-cycle: 
\begin{equation}
\mathrm{Cent}_0/K: \pi' = (i_1i_2\cdots i_m).
\end{equation}
\item The lists $z_1,..,z_k$ and $n_1,...,n_k$ are identical up to rearrangement. 
\item There exists an integer $z$ so that $\pi_0^m=\tau_1^z...\tau_m^z$ and $(\pi_0')^m=\tau_{i_1}^z...\tau_{i_m}^z$
\end{enumerate}
\end{lemma}
\begin{proof}
It follows from Remark \ref{frstremark}, \ref{sndremark}, and \ref{trdremark} that each of the conditions are necessary. It remains to prove that they are sufficient.  Look at the list of integers $n_1,..,n_k$. It is clear that 
\begin{equation}
n_{i_1}=n_{i_2}=...=n_{i_m}=0,
\end{equation}
by the image of $\pi'$ under $\Cent_0(\sigma)/K$. By the transitive action on the cycles $\{\tau_1,\tau_2,..,\tau_k\}$ and the fact that the lists $z_1,..,z_k$ and $n_1,...,n_k$ are identical up to rearrangement there exists some $\rho \in \Cent_0(\sigma)$ such that 
\begin{equation}
\rho\mu_0'\rho^{-1}=\mu_0 \; \; \mathrm{and} \; \; \mathrm{Cent_0(\sigma)}/K: \rho\pi'\rho^{-1} \mapsto (12\cdots m).
\end{equation}
 We are done if we can prove that there exists some integers $j_1,j_2,...,j_m$ so that 
\begin{equation}
(\tau_{1}^{j_1}\tau_{2}^{j_2}...\tau_{m}^{j_m})(\rho\pi_0'\rho^{-1})(\tau_{1}^{j_1}\tau_{2}^{j_2}...\tau_{m}^{j_m})^{-1} = \pi_0.
\end{equation}
Set $\rho\pi_0'\rho^{-1}=\pi_0''$. By condition (iii), there exists some integer $z$ so that 
\begin{equation}
(\pi_0)^m=(\pi_0'')^m=\tau_1^z\tau_2^z...\tau_m^z.
\end{equation}
This implies that $\pi_0, \pi_0''$ have the same cycle type. Let $d'$ denote the length of the cycles in $\tau_1^z$. Then $\pi_0, \pi_0''$ have cycle type $\frac{d}{d'}$ disjoint cycles of length $md'$. 

We observe that the elements $\pi_0, \pi_0''$ are completely determined by one of their cycles. To see this, let $c_0$ denote a cycle in $\pi_0$ and $c_0''$ denote a cycle in $\pi_0''$. Then 
\begin{equation}
\begin{split}
c_0& = (x_{\tau_1}\cdots x_{\tau_m}x_{\tau_1+z}\cdots x_{\tau_m+z}\cdots x_{\tau_1+d'z}\cdots x_{\tau_m+d'z}) \\
 c_0''& = (y_{\tau_1}\cdots y_{\tau_m}y_{\tau_1+z}\cdots y_{\tau_m+z}\cdots y_{\tau_1+d'z}\cdots y_{\tau_m+d'z}),
\end{split}
\end{equation}
where $x_{\tau_i}, y_{\tau_i}$ are points in the cycles $\tau_i$ , $\tau_i^{hz}(x_{\tau_i})=x_{\tau_i+hz}$, and $\tau_i^{hz}(y_{\tau_i})=y_{\tau_i+hz}$. These relations completely determine how the remaining points in the cycles $\tau_1,...\tau_m$ are sent. We see that $\pi, \pi''$ are not only determined by the cycles $c_0, c_0''$ but further, completely determined by the first $m$ points in the cycles. Naturally, there exists integers $h_1,..,h_m$ so that
\begin{equation}
\tau_i^{h_i}(x_{\tau_i})=y_{\tau_i}, \; \; 1 \leq i \leq m, \end{equation}
which concludes our argument. 
\end{proof}

\begin{corollary}
\label{relprimecor}
Assume that the order of $\sigma$, $|\tau_1...\tau_k|=d$ and the order of $\pi$ are relatively prime, then $\pi$ and $\pi'$ are conjugate if $\pi, \pi'$ are conjugate in $S_n$. 
\end{corollary}
\begin{proof}
By $\pi, \pi'$ having order relatively prime to $d$, we must have that $\mu = \mu' =1$. Also $\pi^m=1$ and since $\pi'$ conjugate to $\pi$ and relatively prime to $d$ , the image of $\pi'$ is also an $m$-cycle and $(\pi')^m=1$. The result follows from Lemma \ref{importantabelianlemma}.
\end{proof}

\begin{corollary}
\label{transpocase}
Assume that the order $\sigma$, $d=2$. The elements $\pi,\pi'$ are conjugate in $\Cent_0(\sigma)$ if $\pi,\pi'$ are conjugate in $S_n$ and have the same cycle type under $\Cent_0(\sigma)/K$
\end{corollary}
\begin{proof}
If $\pi, \pi'$ have odd order, the result follows from the corollary above. Assume $\pi, \pi'$ have even order. $\pi', \pi_0'$ act non-trivially on $md$ points. It is clear that $\pi, \pi'$ have the same number of fix points and so we can assume that there exists $\rho \in \Cent_0(\sigma)$ so that $\rho\mu'\rho^{-1}=\mu$ and $\rho\pi_0'\rho^{-1}, \pi$ have the same image under $\Cent_0(\sigma)/K$. The conjugacy in $S_n$ further forces that $(\pi_0')^m=(\pi_0)^m=1$ or $(\pi_0')^m=(\pi_0)^m=\tau_1...\tau_k$. The result now follows from \ref{importantabelianlemma}.
\end{proof}

Note that Corollary \ref{relprimecor} and \ref{transpocase} can easily be generalized to the case where $\pi$ has arbitrary cycle type under $\Cent_0(\sigma)/K$. 

\subsection{Centralizer Actions for Dihedral Source}

Assume as before that $\sigma \in S_n$ is the product of $k$ disjoint cycles of length $d$, $\tau_1\tau_2...\tau_k$ and $\mathbf{Fix}(\sigma)= \emptyset$. We are interested in exploring elements $\pi, \pi' \in S_n$ of the type 
\begin{equation}
\pi^2 = (\pi')^2 = 1 \; \; \mathrm{and} \; \; \pi\sigma\pi^{-1} = (\pi')\sigma(\pi')^{-1} = \sigma^{-1}. 
\end{equation}
We start by noting that $\pi, \pi'$ are elements in a larger subgroup than $\Cent_0(\sigma)$, $H_{\sigma}$, where we also have a transitive action on the cycles $\{\tau_1,\tau_2,...,\tau_k\}$:
\begin{lemma}
\label{actiondihedral}
The subgroup $H_{\sigma} = \{  x \in S_n | x\sigma x^{-1} = \sigma^{z}, \; z\in \mathbb{Z} \}$ acts transitively on the cycles $\{\tau_1,\tau_2,...,\tau_k\}$. The kernel $K_{H_\sigma}$ of this action is 
\begin{equation}
K_{H_\sigma} = \{ x \in H_\sigma | x\tau_ix^{-1} = \tau_i^z, \; \; z \in \mathbb{Z} \; \mathrm{for} \; \mathrm{all} \; 1 \leq i \leq k \}.
\end{equation}
\end{lemma}
\begin{proof}
This is straightforward to prove, see Lemma \ref{actioncycles}. 
\end{proof}

It is obvious that $\Cent_0(\sigma) \leq H_\sigma$, and in fact $\Cent_0(\sigma) \unlhd H_\sigma.$  As for the non-trivial centralizer, we get a permutation representation, this time from $H_\sigma$ to $S_k$: 
\begin{equation}
H_{\sigma} \to H_{\sigma}/K_{H_\sigma} \cong S_k. 
\end{equation}
The main goal of this subsection is to prove that if $\pi, \pi'$ are conjugate in $S_n$ and if the images of the elements under $H_{\sigma}/K_{H_\sigma}$ are also conjugate in $S_k$, then there exists some $\rho \in \Cent_0(\sigma)$ so that 
\begin{equation}
\rho (\pi') \rho^{-1} = \pi.
\end{equation}

We first look at elements $s_1, s_2 \in S_{X_{\tau_1}}$, that send a single cycle $\tau_1$ to its own inverse or elements $s_1,s_2 \in S_{X_{\tau_1,\tau_2}}$ that sends the cycle $\tau_1$ to the inverse of a distint cycle $\tau_2$. 

\begin{lemma}
\label{dihed1}
Assume $\tau_1 = (12\cdots d)$ and $\tau_2 = (j_1j_2 \cdots j_d)$. Let $s_1, s_2 \in S_{X_{\tau_1,\tau_2}}$ both be products of disjoint transpositions such that 
\begin{equation}
s_1\tau_1s_1^{-1} = s_2\tau_1s_2^{-1}=\tau_2^{-1}. 
\end{equation}
There exists an integer $z$ such that 
\begin{equation}
\tau_2^{z}s_1\tau_2^{-z} = s_2.
\end{equation}
\end{lemma}
\begin{proof}
There are precisely $d$ ways to send the cycle $\tau_1$ to $\tau_j^{-1}$ through transpositions. Once we send the point 1 in $\tau_1$ to some point in $\tau_2$, the element is determined. Assume $s_1 = (1j_1)...$ and $s_2 = (1j_i)...$. Then 
\begin{equation}
\tau_2^{i-1}s_1\tau_2^{-(i-1)}=s_2.
\end{equation}
\end{proof}

\begin{lemma}
\label{dihed2}
Let $\tau_1=(12\cdots d)$ and let $s_1,s_2 \in S_{X_{\tau_1}}$ be products of disjoint transpositions such that 
\begin{equation}
s_1\tau_1s_1^{-1}=s_2\tau_1s_2^{-1}=\tau_1^{-1}.
\end{equation}
If $s_1,s_2$ are conjugate in $S_n$, there exists some integer $z$ so that: 
\begin{equation}
\tau_1^z s_1\tau_1^{-z}=s_2.
\end{equation}
\end{lemma}
\begin{proof}
The elements $s_1,s_2$ may precisely be associated with reflections along the lines of symmetry of a regular polygon with $d$ vertices. If $d$ is odd, $s_1,s_2$ will have cycle type $\frac{d-1}{2}$ transpositions. The lines of symmetry of a regular polygon with an odd number of vertices is completely determined by the vertex it intersects. Assume $s_1$ fixes vertex 1 and $s_2$ fixes vertex $i$. Then $\tau_1^{i-1}s_1\tau_1^{-(i-1)}=s_2$. 

If $d$ is even, a line of symmetry either intersects two vertices or an edge. In the former case $s_1,s_2$ have cycle type $\frac{d-2}{2}$ and the elements are again completely determined by a fixed vertex so that if $s_1$ fixes vertex 1 and $s_2$ fixes vertex i, then $\tau_1^{i-1}s_1\tau_1^{-(i-1)}=s_2$.  

In the latter case, $s_1,s_2$ have cycle type $\frac{d}{2}$ transpositions and the elements are determined by an edge that the line of symmetry intersects. Let $s_1 = (12)...$ and $s_2=(i \; i+1)...$ , then $\tau_1^{i-1}s_1\tau_1^{-(i-1)}=s_2$ and we are done. 
\end{proof}
We are ready to state the main lemma:

\begin{lemma}
\label{importantdihedlemma}
Let $\pi,\pi'$ be elements conjugate in $S_n$ and assume that their images under $H_{\sigma}/K_{H_{\sigma}}$ have the same cycle type. Then, there exists some $\rho \in \Cent_0(\sigma)$ so that $\rho\pi\rho^{-1}=\pi'$.  
\end{lemma}

\begin{proof}
We may write $\pi,\pi'$ in the form $\pi=\mu\pi_0$, $\pi'=\mu'\pi_0'$ where $\mu, \mu' \in K_{H_{\sigma}}$ and $\pi_0, \pi_0'$ either fix all points in a cycle $\tau_i$ or sends all the points to a distinct cycle $\tau_j$. Assume $\mu$ acts non-trivially on the points in the cycles $\tau_1,...,\tau_m$. Then, $\pi_0$ acts non-trivially on the cycles $\tau_{m+1},...,\tau_k$. Since $\pi, \pi'$ have the same cycle type under $H_{\sigma}/K_{H_\sigma}$, this implies that $\mu, \mu'$ are conjugate in $S_n$. By the transitive action on the cycles $\tau_1,..,\tau_k$ in $\Cent_0(\sigma)$, this in turn implies that there exists some $\rho \in \Cent_0(\sigma)$ so that $\rho \mu' \rho^{-1}$ acts non-trivially on the cycles $\tau_1,...,\tau_m$ and  $\rho\pi_0'\rho^{-1}, \rho\pi\rho^{-1}$ have the same image under $H_\sigma/K_{H_{\sigma}}$. The result now follows from Lemma \ref{dihed1} and \ref{dihed2}. 
 \end{proof}

 \begin{corollary}
 \label{odddihedral}
 If the order of $\sigma$, $|\tau_1...\tau_k|=d$ is odd and $\pi, \pi'$ conjugate in $S_n$, there exists some $\rho \in \Cent_0(\sigma)$ so that $\rho \pi \rho^{-1} = \pi'$
 \end{corollary}
 \begin{proof}
 If $d$ is odd and $\pi, \pi'$ conjugate in $S_n$, this forces $\pi, \pi'$ to have the same image under $H_{\sigma}/K_{H_{\sigma}}$. The result follows from Lemma \ref{importantdihedlemma}.
 \end{proof}

\section{Homomorphism Conjugacy in the Symmetric Group}
Let $\varphi, \psi$ be two generator-conjugate homomorphisms from a group $G$ to $S_n$ with respect to the generators $g_1,....,g_n$. We first wish to determine when $\varphi, \psi$ are conjugate as homomorphisms. We may assume that 
\begin{equation}
\varphi(g_1)=\psi(g_1),
\end{equation}
so that if $\varphi, \psi$ are conjugate, there exists some $\rho \in \Cent_{S_n}(\varphi(g_1))$ so that 
\begin{equation}
\rho\varphi\rho^{-1}=\psi.
\end{equation}
This observation, in conjunction with the results derived about centralizer actions from the previous section, gives the first part of our main result: determining when generator-conjugate homomorphisms to $S_n$ are conjugate for an abelian or dihedral source group with two generators. The second part of the main result, describing the relationship between element-conjugate and conjugate homomorphisms for such source groups, is derived by first studying the relationship between generator-conjugate and element-conjugate homomorphisms through centralizer actions in $S_n$.
\subsection{Homomorphisms Conjugacy for Abelian Source}
Let $\varphi,\psi$ be two homomorphisms from an abelian group $A$, generated by two elements $a,b$ to the symmetric group $S_n$.

Assume $\varphi,\psi$ are generator-conjugate with respect to $a,b$ and that
\begin{equation}
\varphi(a)=\psi(a).
\end{equation}
We may set $\varphi(a)=\sigma$ and rewrite the element so that $\sigma=\sigma_1...\sigma_l$ where $\sigma_i$ is a product of $k_i$ disjoint cycles of length $d_i$, $\tau_{(1,i)}\tau_{(2,i)}...\tau_{(k_i,i)}$, and the integers $d_1,...,d_l$ are all distinct. By the direct product decomposition of centralizers in the symmetric group (Lemma \ref{directproduct}), we may also rewrite $\varphi(b), \psi(b)$:
\begin{equation}
\begin{split}
\varphi(b) &= \mu_0\pi_1\pi_2...\pi_l \\ 
\psi(b) &= \mu_0'\pi_1'\pi_2'...\pi_l',
\end{split}
\end{equation}
where $\mu_0,\mu_0' \in S_{\mathbf{Fix(\sigma}}$ and $\pi_i,\pi_i' \in \Cent_0(\sigma_i)$. By Remark \ref{decomp}, each $\pi_i,\pi_i'$ may decomposed as products: $\pi_i = \mu_i\pi_{0,i}$ and $\pi_i' = \mu_i\pi_{0,i}'$ where $\mu_i,\mu_i' \in K_i = \langle \tau_{1,i},...,\tau_{k_i,i}\rangle$ and $\pi_i, \pi_i'$ either fix all points in a cycle $\tau_{j,i}$ or sends all points to some distinct cycle $\tau_{j',i}$. Note that we have that 
\begin{equation}
\begin{split}
\mu_i &= \tau_{1,i}^{z_{(1,i)}}\tau_{2,i}^{z_{(2,i)}}...\tau_{k_i,i}^{z_{(k_i,i)}} \\
\mu_i' &= \tau_{1,i}^{z_{(1,i)}'}\tau_{2,i}^{z_{(2,i)}'}...\tau_{k_i,i}^{z_{(k_i,i)}'},
\end{split}
\end{equation}
for some integers $z_{(1,i)},..,z_{(k_i,i)}$ and $z_{(1,i)}',..,z_{(k_i,i)}'$.

The action in $\Cent_0(\sigma_i)$ gives a further decomposition of the elements $\pi_{(0,i)}$ and $\pi_{(0,i)}'$ as products of elements in the non-trivial centralizer: 
\begin{equation}
\begin{split}
    \pi_{(0,i)} &= x_{(1,i)}....x_{(h,i)} \\ 
    \pi_{(0,i)}' &= y_{(1,1)}....y_{(h',i)},
\end{split}
\end{equation}
where the image of every $x_{(j,i)},y_{(j,i)}$ is a single cycle under $\Cent_0(\sigma_i)/K_i$. 
\begin{theorem}
\label{abeliantheorem}
The homomorphisms $\varphi, \psi$ are conjugate if and only if: 
\begin{enumerate}[(i)]
\item $\mu_0, \mu_0'$ are conjugate in $S_n$.
\item The lists $z_{(1,i)},..,z_{(k_i,i)}$ and $z_{(1,i)}',..,z_{(k_i,i)}'$are identical up to rearrangement for $1\leq i \leq l$ . 
\item For every $x_{(j,i)}$ whose image is an $m$-cycle under $\Cent_0(\sigma_i)/K_i$ there exists some $y_{(j',i)}$ whose image is also an $m$-cycle under $\Cent_0(\sigma_i)/K_i$ where $x_{(j,i)}^m = \tau_{t,i}^{z_{j,i}}...$ and $y_{j'}^m=\tau_{t',i}^{z_{j,i}}...$ for some integer $z_{j,i}$.
\end{enumerate}
\end{theorem}
\begin{proof}
By the direct product decomposition of the centralizer of $\varphi(a)=\psi(a)=\sigma$ we must assume that there exists some $\rho_0 \in S_{\mathbf{Fix}(\sigma)}$ so that $\rho_0\mu_0'\rho_0=\mu_0$ and $\rho_i \in \Cent_0(\sigma_i)$ so that $\rho_i\pi_i'\rho_i=\pi_i$ for $1\leq i \leq l$. Condition (i) follows from that cycle type determines conjugacy in symmetric groups. Condition (ii) follows directly from Lemma \ref{importantabelianlemma}. Condition (iii) follows from the decomposition of $\pi_{0,i},\pi_{0,i}'$ into disjoint elements whose image under $\Cent_0(\sigma_i)/K_i$ is a single cycle and Lemma \ref{importantabelianlemma}.
\end{proof}
The following two corollaries follow quite easily from Theorem \ref{abeliantheorem} and generalizing the results in Corollary \ref{relprimecor} and \ref{transpocase}. 
\begin{corollary}
\label{relprime}
If the order of $a,b$ are relatively prime then $\varphi, \psi$ are conjugate if and only if: $\mu_0,\mu_0'$ conjugate in $S_n$ and $\pi_i, \pi_i'$ are conjugate in $S_n$ for $1\leq i \leq l$.
\end{corollary}

\begin{corollary}
If  $|a|=2$ then $\varphi, \psi$ are conjugate if and only if: $\mu_0,\mu_0'$ conjugate in $S_n$, $\pi_i, \pi_i'$ are conjugate in $S_n$, and the elements also have the same cycle type under $\Cent_0(\sigma_i)/K_i$ for $1\leq i \leq l$.
\end{corollary}

Note that although the source group is cyclic in Corollary \ref{relprime}, the result is not trivial. We can easily define generator-conjugate homomorphisms from cyclic source groups that are not conjugate. 

\begin{ex}
Consider the homomorphisms $\varphi, \psi: A \cong \Z/3\Z \times \Z/2\Z \to S_{12}:$
\begin{equation}
\begin{split}
&\varphi, \psi: a \mapsto (123)(456); \\
&\varphi: b \mapsto (14)(25)(36), \; \psi: b \mapsto (78)(9\;10)(11\;12).
\end{split}
\end{equation}
\end{ex}
The homomorphisms are clearly generator-conjugate, but one can check that $\varphi(ab), \psi(ab)$ have different cycle types in $S_n$. 

Further, the conditions set out in Theorem \ref{abeliantheorem} make it easy to define distinct generator-conjugate homomorphisms from a non-cyclic abelian group to $S_n$.
\begin{ex}
Consider the homomorphisms $\varphi, \psi: A \cong \Z/4\Z \times \Z/2\Z \to S_{11}:$
\begin{equation}
\begin{split}
&\varphi, \psi: a \mapsto (1234)(5678); \\
&\varphi: b \mapsto (15)(26)(37)(48)(9\;10), \; \psi: b \mapsto (16)(27)(38)(45)(10\;11).
\end{split}
\end{equation}
\end{ex}
These homomorphisms are conjugate by Theorem \ref{abeliantheorem} and it is straightforward to check they are distinct. 

We conclude this subsection by describing the relationship between element-conjugate and conjugate homomorphisms: 

\begin{corollary}
\label{localglobalabelian}
The homomorphisms $\varphi, \psi$ are conjugate if and only if they are element-conjugate on restriction to $S_{\mathrm{Fix}(\sigma)}$ and $\Cent_0(\sigma_i)$, for $1 \leq i \leq l$.  
\end{corollary}

If $\varphi, \psi$ are conjugate as homomorphisms, it is clear by the direct product decomposition of the centralizer that we must assume that the homomorphisms are also element-conjugate when restricted to $S_{Fix(\sigma)}$ and each subgroup $\Cent_0(\sigma_i)$. It is then relevant to restrict ourselves to investigating when generator-conjugate homomorphisms are element-conjugate in the case where $\psi(a)=\psi(a)$ is the product of $k$ disjoint cycles of length $d$, $\tau_1...\tau_k$ and this element has no fix points in $S_n$. Recall that in this case, we can decompose $\varphi(b), \psi(b)$ such that: 
\begin{equation}
\begin{split}
    \varphi(b) &= \mu x_{1}....x_{h} \\ 
    \psi(b) &= \mu'y_{1}....y_{h'},
\end{split}
\end{equation}
where $\mu, \mu' \in K = \langle\tau_1,...,\tau_k\rangle$ and the images of all $x_j,y_j$ under $\Cent_0(\sigma)/K$ are single cycles and $x_i,y_i$ do not  non-trivially send points in a cycle to itself. We also have that $\mu = \tau_1^{z_1}...\tau_k^{z_k}$ and $\mu' = \tau_1^{n_1}...\tau_k^{n_k}$ for some integers $z_1,...,z_k$ and $n_1,...,n_k$.

We present an important lemma that relates generator-conjugacy to element-conjugacy:

\begin{lemma}
Assume $\psi(a)=\varphi(a)= \tau_1....\tau_k$ and also that $\mathbf{Fix}(\tau_1...\tau_k)=\emptyset$. If $\varphi,\psi$ are generator-conjugate, they are element-conjugate if and only if: 
\begin{enumerate}[(i)]
\item The lists $z_{i},..,z_{k}$ and $n_{1},..,n_{k}$ are identical up to rearrangement. 
\item For every $x_{j}$ whose image is an $m$-cycle under $\Cent_0(\tau_1...\tau_k)/K$ there exists some $y_{j'}$ whose image is also an $m$-cycle under $\Cent_0(\sigma_i)/K_i$ and $x_{j,}^m = \tau_{i}^{z}...$, $y_{j'}^m=\tau_{i'}^{z}...$ for some integer $z$.  
\end{enumerate}
\end{lemma}
\begin{proof}
If $\varphi,\psi$ are element-conjugate, $\varphi(a^{h_1}b^{h_2})$ and $\psi(a^{h_1}b^{h_2})$ have the same number of fix points for all integers $h_1,h_2$. This immediately implies that the lists $z_{i},..,z_{k}$ and $n_{1},..,n_{k}$ are identical up to rearrangement. Ordering the elements $x_1,...,x_h$ after their cycle length under $\Cent_0(\sigma)$, $m_1\leq m_2 \leq ...\leq m_h$ , the rest of the proof follows from comparing fix points and a simple induction argument. 
\end{proof}

Note that this lemma is precisely the conditions required for homomorphism conjugacy in this case, and Corollary \ref{localglobalabelian} follows.

\subsection{Homomorphism Conjugacy for Dihedral Source}
Let $\varphi,\psi$ be homomorphisms from the dihedral group $D_{2m}$ to $S_n$ and assume the homomorphisms are generator-conjugate with respect to the generators $r,s$:
\begin{equation}
D_{2m} = \langle r,s \; | \; r^m=s^2 = 1 \; srs = r^{-1} \rangle.
\end{equation}
We assume $\varphi(r)=\psi(r)=\sigma$ and rewrite $\sigma$ in the form $\sigma=\sigma_1...\sigma_l$ where $\sigma_i$ is the product $k_i$ disjoint cycles of length $d_i$ and all the integers $d_1,..,d_l$ are distinct. By the invariance of cycle type under conjugation, the elements $\varphi(s),\psi(s)$ can be written as products: 
\begin{equation}
\begin{split}
\varphi(s) &= \mu \pi_{1}....\pi_{l} \\ 
    \psi(s) &= \mu'\pi_{1}'....\pi_{l}',
\end{split}
\end{equation}
where $\mu, \mu' \in S_{\mathbf{Fix}(\sigma)}$ and $\pi_i, \pi_i' \in H_{\sigma_i}$ recalling that
\begin{equation}
    H_{\sigma_i} = \{ x \in S_{X_{\sigma_i}} | \; x\sigma_ix=\sigma_i^z, \; z \in \Z \}.
\end{equation}
By previous results, the subgroup $H_{\sigma_i}$ acts on the cycles in $\sigma_i$ through the permutation representation $H_{\sigma_i}/K_{\sigma_i}$ (Lemma \ref{actiondihedral}).
\begin{theorem}
\label{dihedraltheorem}
The homomorphisms $\varphi, \psi$ are conjugate if and only if:
 \begin{enumerate}[(i)]
\item $\mu_0, \mu_0$ are conjugate in $S_n$.
\item $\pi_i,\pi_i'$ are conjugate in $S_n$ and the elements also have the same cycle type under $H_{\sigma_i}/K_{H_{\sigma_i}}$, $1\leq i \leq l$.
\end{enumerate}
\end{theorem}
\begin{proof}
The direct product decomposition of $\varphi(s),\psi(s)$ due to invariance of cycle type and Lemma \ref{importantdihedlemma}.
\end{proof}
We finally note that the relationship between element-conjugacy and homomorphism conjugacy for the dihedral case is similar to the abelian case with two generators: 
\begin{corollary} 
\label{localglobaldihedral}
The homomorphisms $\varphi, \psi$ are conjugate if and only if they are element-conjugate on restriction to $S_{\mathrm{Fix}(\sigma)}$ and $H_{\sigma_i}$, for $1 \leq i \leq l$.  
\end{corollary}
\begin{proof}
If the order of $r$ is odd, the result follows from Corollary \ref{odddihedral}. If $d$ is even, we proceed by contradiction to show that if $\pi_i, \pi_i'$ have distinct cycle types under $H_{\sigma_i}/K_{H_{\sigma_i}}$, on restriction to $H_{\sigma_i}$ the elements $\varphi(rs), \psi(rs)$ are not conjugate in $S_n$. 
\end{proof}
\bibliography{references}
\end{document}